\documentclass[a4paper, 11pt]{article}
\usepackage{amsmath, amssymb,amscd, amsthm}

\usepackage[mathscr]{eucal}
\usepackage{fullpage}
\usepackage{url}
\usepackage{cancel}
\usepackage[margin=0.95in]{geometry}
\usepackage{ulem}

\newcommand\cyr{%
 \renewcommand\rmdefault{wncyr}%
 \renewcommand\sfdefault{wncyss}%
 \renewcommand\encodingdefault{OT2}%
\normalfont\selectfont} \DeclareTextFontCommand{\textcyr}{\cyr}

\newtheorem{theorem}{Theorem}
\newtheorem{lemma}[theorem]{Lemma}

\newtheorem{proposition}[theorem]{Proposition}
\newtheorem{remark}[theorem]{Remark}

\DeclareMathOperator{\sign}{sign}

\long\def\symbolfootnote[#1]#2{\begingroup%
\def\thefootnote{\fnsymbol{footnote}}\footnote[#1]{#2}\endgroup}

\title{An explicit two-source extractor with min-entropy rate near $4/9$}

\author{Mark Lewko}

\date{}

\begin{document}

\maketitle

\begin{abstract}In 2005 Bourgain gave the first explicit construction of a two-source extractor family with min-entropy rate less than $1/2$. His approach combined Fourier analysis with innovative but inefficient tools from arithmetic combinatorics and yielded an unspecified min-entropy rate which was greater than $.499$. This remained essentially the state of the art until a 2015 breakthrough of Chattopadhyay and Zuckerman in which they gave an alternative approach which produced extractors with arbitrarily small min-entropy rate.

In the current work, we revisit the Fourier analytic approach. We give an improved analysis of one of Bourgain's extractors which shows that it in fact extracts from sources with min-entropy rate near $\frac{21}{44} =.477\ldots$, moreover we construct a variant of this extractor which we show extracts from sources with min-entropy rate near $4/9 $ = $.444\ldots$. While this min-entropy rate is inferior to Chattopadhyay and Zuckerman's construction, our extractors have the advantage of exponential small error which is important in some applications. The key ingredient in these arguments is recent progress connected to the restriction theory of the finite field paraboloid by Rudnev and Shkredov. This in turn relies on a Rudnev's point-plane incidence estimate, which in turn relies on Koll\'ar's generalization of the Guth-Katz incidence theorem. 
\end{abstract}

\symbolfootnote[0]{2010 Mathematics Subject Classification 42B10}

\section{Introduction}

A well-known probability puzzle asks how to simulate a fair coin toss given a coin with unknown bias. A solution, due to Von Neumann \cite{V}, is to flip the coin twice and if the flips disagree use the first, otherwise flip twice again and repeat. We leave it to the reader to verify that this algorithm does in fact solve the problem. The purpose of extractors in theoretical science is to address a similar problem: Given the output of a random variable/source with unknown biases, subject to only some weak assumptions, use these outputs to simulate a random variable with much stronger (more uniform) randomness properties.

We say that a real-valued random variable $X:\Omega \rightarrow R$ has min-entropy $k$ if $\sup_{r \in R}\mathbb{P}[X = r] \leq 2^{-k}.$ Given random variables $X$ and $Y$ we define the statistical distance between them as
$$SD(X,Y) := \frac{1}{2} \sum_{r\in R} \left| \mathbb{P}[X=r] - \mathbb{P}[Y=r] \right|.$$ 

We call a family of maps $Ext_N : [N] \times [N] \rightarrow \{0,1\}$ an extractor family with min-entropy $\rho$ if there exists a fixed $\delta >0$ such that for any independent random variables $X,Y: \Omega \rightarrow [N]$ with min-entropy rate $\rho \log_2 N$ one has
$$ SD \left(Ext_N(X,Y), U \right) \lesssim \log^{-\delta} N.$$
where $U$ is the uniform distribution on $\{0,1\}$. To ease notation, we will say that $Ext_N$ is an extractor family with min-entropy rate \textit{near} $\rho$ if it extracts from all sources with min-entropy rate $\rho' > \rho$. We will say that $Ext_N$ has exponentially small error if one has the stronger estimate
$$ SD \left(Ext_N(X,Y), U \right) \lesssim  N^{-\delta}$$
for some fixed $\delta >0$.

\begin{remark}The computer science literature typically considers maps from $\{0,1\}^\ell \times \{0,1\}^\ell \rightarrow \{0,1\}$ when defining extractors. One can convert to this regime by taking $\ell = \left \lceil{2 \log p}\right \rceil$.
\end{remark}

There are a number of elementary constructions of extractor families with min-entropy rate greater than $1/2$. In 2005, Bourgain \cite{B} gave the first explicit constructions of an extractor with min-entropy rate less than $1/2$. His proof depended on the finite field Szemer\'di-Trotter theorem \cite{BKT} and gave an unspecified min-entropy rate. An optimization of the key ingredient, the Szemer\'edi-Trotter incidence theorem, by Helfgott and Rudnev \cite{HR} in 2010 gave an exponent $\frac{3}{2} - \frac{1}{10678}$ in that result. Combined with Bourgain's method in the form below, this gives a min-entropy rate estimate of $\frac{16017}{32036} = .49996\ldots$. The incidence estimate was further refined by Jones in 2011 who gave an incidence estimate with exponent of $\frac{3}{2}- \frac{1}{622} = \frac{496}{331}$ which gives a  min-entropy rate of around $\frac{2647}{5296}=.4998\ldots$. More recently, Stevens and de Zeeuw \cite{Sd} introduced several new ideas which allowed them to improve the exponent in the incidence theorem to $\frac{3}{2}- \frac{1}{30} = \frac{22}{15}$. This yields an extractor with min-entropy rate near $\frac{45}{92}=.489\ldots$. Lastly we note that an optimal Szemer\'edi-Trotter theorem in finite fields with exponent $4/3$ would imply Bourgain's extractor has min-entropy rate near $\frac{9}{20}=.45$, following the initial approach of \cite{B}. See also \cite{HH} for generalizations of Bourgain's work.

In a 2015 breakthrough paper, Chattopadhyay and Zuckerman \cite{CZ} gave an alternative approach which produced explicit extractors which extract from sources with arbitrarily small (in fact poly-logarithmic) min-entropy rate.  One advantage to Bourgain's Fourier analytic constructions is that these extractors have exponentially small error, where Chattopadhyay and Zuckerman's only have poly-logarithmically small error. Finding explicit constructions with both an arbitrarily small min-entropy rate and exponentially small error remains an important open problem in the field. Our interest here is to show how one can push Bourgain's Fourier analytic approach further using recent advances in finite field incidence geometry. This allows us to produce explicit constructions with exponentially small error and a lower min-entropy rate than available in the prior literature.

We start by recalling one of Bourgain's construction. Let $F$ be a prime order finite field. Given an element $x \in F$, we let $\sigma(x) \in [0,\frac{p-1}{p} ]$ denote the real numbers in the interval $[0,1]$ given by considering the integer representation of $x$ between $0$ and $p-1$ and dividing that number by $p$. Moreover we let $\rho : F \rightarrow \{0,1\}$ be defined by $\rho(x) := \sign \sin \left( 2 \pi \sigma(x) \right)$ where we adopt the convention that $\sign(0) = 1$. We take $N=p^2$ and consider the maps $Ext_{p^2}$ from $F^2 \times F^2 \rightarrow \{0,1\}$ given by
$$(x,y) \rightarrow \rho (x\cdot y + x \cdot x \times y \cdot y).$$
If $p$ is chosen so that there does not exists an element $i \in F$ such that $i^2=-1$, then the map just defined is a Bourgain extractor.
Our first result is a refined analysis of one of Bourgain's extractor:
\begin{theorem}\label{thm:3d}Bourgain's extractor (for $F$ a prime order finite field in which $-1$ is not a square) defined above extracts from two independent sources with min-entropy rate near $\frac{21}{44} =.477\ldots$ and has exponentially small error.
\end{theorem}

Next we show that one can construct extractors with lower min-entropy rate by replacing $F^2$ in Bourgain's construction with $F^3$. More precisely:

\begin{theorem}\label{thm:4d}Let $F$ be finite field of order $p$ and identify $[N] = F^3$. Then the maps $Ext_{p^3} (x,y) \rightarrow \rho(x\cdot y + x \cdot x \times y \cdot y)$ from $F^3 \times F^3 \rightarrow \{0,1\}$ is a two-source extractor with min-entropy rate near $\frac{4}{9} =.444\ldots$ and has exponentially small error.
\end{theorem}

We note that Bourgain's extractor requires one to restrict to prime order fields in which $-1$ is not a square, while this hypothesis on the prime order field in Theorem \ref{thm:4d}. In Bourgain's case if $F$ is such that there exists an $i \in F$ such that $i^2=-1$ then the the zero set of $z:=\{x \in F^2 : \sigma(x)=0\}$ contains the line $\{(t, it) : t \in F \} \subset F^2$ and the required additive energy will not hold for certain subsets of this line. In the second case, the analogous level set will contain a line regardless of the field. However to prove Theorem \ref{thm:4d} one only needs control of the additive energy for sets substantially larger than that of a line. This phenomenon is explained by the theory of quadratic forms in finite fields, and is discussed in detail in \cite{LewkoKakeya} and also is connected to the recent observation that one can obtain sharp $L^2$  restriction theorems for the finite field paraboloid in high dimensions while the analogous problem remains open in $3$ dimensions. See \cite{IKL} and \cite{RS}.

The proofs proceed by observing that Bourgain's argument can be extended to a rather general statement relating extraction properties of certain maps to the additive energy of certain subsets in finite fields. Lemma \ref{lem:reduct} and  \ref{lem:expSum} below articulates this generalization. The second step is to observe that these statements reduce the maps above to additive energy estimates for subsets of the paraboloid obtained by Rudnev and Shkredov \cite{RS} in their recent work on the finite field restriction problem. See \cite{IKL}, \cite{LewkoEnd}, \cite{LewkoKakeya}, \cite{LewkoImp}, and \cite{MT} for a discussion of this problem.

Finally we remark that inserting an optimal finite field Szemer\'edi-Trotter theorem into the Rudnev-Shkredov \cite{RS} machinery yields an additive energy estimate which when inserted into the machinery below shows that Bourgain's extractor defined above extracts from sources with min-entropy rate near $3/8 = .375.$

\textbf{Acknowledgment} We thank David Zuckerman for comments on an earlier draft of this note.

\section{Notation and Preliminaries}
We will use $R$ and $C$ to denote the fields of real and complex numbers, respectively. We will use $F$ to denote a finite field, which will always be of prime order. We denote the non-zero elements of $F$ by $F_*$. As usual we write the additive character on $F$ as $e(x):=e^{2\pi i x /p}$. Given $x,y \in F^{n}$ we write $x\cdot y := x_1 y_1 + \ldots+ x_n y_n$. Given a subset $A \in F^{n}$ we define the additive energy $\Lambda(A):= \sum_{\substack{a+b=c+d \\ a,b,c,d \in A}} 1 = \sum_{x \in F^n} \left( \sum_{\substack{a+b=x \\ a,b \in A}} 1 \right)^{2}.$ Parsavel's identify is the equality
$$ \sum_{x \in F^n} \left| \sum_{\xi \in F^n} f(\xi) e(x \cdot \xi) \right|^2 = |F| \sum_{\xi \in F^n} |f(\xi)|^2.$$
If $f$ is a real or complex valued function on a domain $D$ we define $|| f ||_{\ell^\infty} := \sup_{x\in D} |f(x)|$. We will write $X \sim Y$ to indicate that $Y/2 \leq X \leq 2Y$. Typically we use this to select a level set of a function. For example $D_{\lambda } := \{x \in D : f(x) \sim \lambda\}$ would denote the elements of the domain, $D$, of $f$ where $\lambda/2 \ \leq f(x) \leq 2 \lambda $. We will also use the notation $a \lesssim b$ to indicate that the inequality $a \leq c b$ holds with some universal constant $c$. Let $F^d$ denote the $d$ dimensional vector space over $F$. We define the $d$-dimensional paraboloid $P_d \subset{F}^d$ by $P_d:=\{ (\underline{x},\underline{x}\cdot\underline{x}),   \underline{x} : F^{d-1}\}$.

\section{Estimates}

We start by showing how Fourier analytic estimates imply extractor-type properties. Recall that the Fourier coefficients in the expansion $sign \sin(x) = \sum_{\xi \in F} c(\xi) e(\xi x)$ satisfy 
\begin{equation}\label{eq:ssCo} \sum_{\xi \in F} |c(\xi)| \lesssim \log |F|.
\end{equation}
This can be seen from a simple and direct computation with the Dirichlet Kernel, see Remark 3.3 in \cite{B}.

\begin{lemma}\label{lem:reduct}Let $f:F^n \times F^n \rightarrow F$ be a family of maps indexed by $F$, such that the following holds. For some fixed $\delta >0$, one has for all functions $a,b: F^n \rightarrow F$ with $||a||_{\ell^\infty},||b||_{\ell^\infty} \leq 1$ with respective supports $A,B \in F^n$ satisfying $|A|,|B| \geq |F|^{n \rho}$ that  
$$\max_{\lambda \in F_*} |\sum_{x\in A, y\in B} a(x) b(y) e(\lambda f(x,y)) | \lesssim |F|^{-\delta} |A| |B|$$ then the family of maps $ \rho (f(x,y))$ from $F^n \times F^n \rightarrow \{0,1\}$ is a two-source extractor family with min-entropy rate near $\rho$. 
\end{lemma}
\begin{proof}Let $U$ denote the uniform distribution on $\{0,1\}$. Our goal is to show $SD( \rho (f(x,y)), U) \leq |F|^{-\delta}$ for some $\delta >0$. Let $X$ and $Y$ denote independent random variables with min-entropy $k$ and probability mass functions $A$ and $B$, respectively. Let $J':F \times F \rightarrow \{0,1\}$ denote the probability mass function of $\rho (f(x,y))$, that is $J'(x):= \mathbb{P}[ \rho (f(x,y)) =x]$ and $J(x) = J'(x) - 2^{-1}$. We then have
$$ SD( \rho (f(x,y)), U)  = \frac{1}{2} \sum_{x\in \{0,1\}} \left| J(x) \right|.$$
Expanding in a Fourier series $\sign \sin x = \sum_{\lambda \in F} c(\lambda) e(\lambda x)$, and applying the estimate \eqref{eq:ssCo} we have that the above is
$$\lesssim \left| \mathbb{E} \rho (f(x,y)) \right| = \left| \sum_{x \in F^n}\sum_{y \in F^n}\sum_{\lambda \in F} A(x) B(y) c(n) e(\lambda f(x,y) ) \right|$$
$$\lesssim \log |F| \max_{\lambda \in F_*}\left|  \sum_{y \in F^n}\sum_{x \in F} A(x) B(y)  e(\lambda f(x,y) ) \right|.$$
The min-entropy assumption implies that $A(x),B(x) \leq |F|^{-\eta}$. Let $S_\theta := \{x \in F^n : A(x) \sim \theta \}$ and similarly $W_\theta := \{x \in F^n : B(x) \sim \theta \}$. Now we split $A(x)$ and $B(y)$ into $O(\log |F|)$ level sets such that $A_\ell(x) = A(x) \sim 2^{-\ell}$ and  $B_\ell(x) = B(x) \sim 2^{-\ell}$ on their support which is at most $2^{\ell}$. From the hypotheses we may assume that $\ell \geq k$ and have

$$ \max_{\lambda \in F_{*}} |\sum_{x\in A, y\in B} A(x)B(y) e(\lambda f(x,y))| \leq \sum_{i,j=k}^{\log |F| } \max_{\lambda \in F_{*}} |A_i(x) B_j(y) e(\lambda f(x,y))|$$
$$ \lesssim |F|^{-\delta} \sum_{i,j=k}^{\log |F| } 2^{-i-j}|\text{supp}(A_i)| |\text{supp}(B_j)| \lesssim |F|^{-\delta} \log |F| $$
This completes the proof.
\end{proof}
\begin{remark}The above argument is a special case of what is called Vazirani's XOR lemma in the computer science literature. See Lemma 4.1 in \cite{Rao}. This lemma, which considers the more general case where the function $\rho$ is replaced by a map into other finite abelian groups larger than $\{0,1\}$, allows one to use the methods presented here to obtain extractors that output more than one bits. We refer the reader to \cite{Rao} and omit the routine details. 
\end{remark}

Next we show how to obtain estimates on exponential sums of the form appearing in the statement of Lemma \ref{lem:reduct}.$•$
\begin{lemma}\label{lem:expSum}Let $A,B \subseteq F^n$ and $a,b$ functions on $A$ and $B$, respectively, such that $||a||_{\ell^\infty}, ||b||_{\ell^\infty} \leq 1$. Then
$$ \max_{\lambda \in F_*}\left|\sum_{x \in A, b \in B} a(x) b(y) e(\lambda x\cdot y) \right|  \leq |A|^{1/2} |B|^{1/2 } |F|^{n/8} \left( \Lambda(A) \Lambda(B) \right)^{1/8}. $$ 
\end{lemma}

\begin{proof}By Cauchy-Schwarz
$$ \max_{\lambda \in F_*} \left|\sum_{x \in A, y \in B} a(x) b(y) e(\lambda x\cdot y) \right| \leq |A|^{1/2} \max_{\lambda \in F_*}\left( \sum_{x\in A } \left|\sum_{y \in B} a(x) b(y) e(\lambda x\cdot y) \right|^2 \right)^{1/2}$$ 
$$ =  |A|^{1/2} \max_{\lambda \in F_*} \left( \sum_{x\in A } \sum_{y_1, y_2 \in B} a(x) b(y_1) b(y_2)  e\left(\lambda x\cdot (y_1 - y_2) \right)  \right)^{1/2}$$ 
$$ \leq |A|^{1/2}|B|^{1/2} \max_{\lambda \in F_*} \left( \sum_{x_1, x_2 \in A } \sum_{y_1, y_2 \in B} a(x_1)a(x_2) b(y_1) b(y_2)e\left( \lambda (x_1 - x_2) \cdot (y_1 - y_2) \right)  \right)^{1/4}  $$
Now for $\xi, \eta \in F^n$ let 
$$\mathcal{A}_{\lambda}(\xi) = \sum_{\substack{ x_1, x_2 \in A \\ x_1- x_2 = \lambda^{-1} \xi}} a(x_1)a(x_2),$$   
$$\mathcal{B}(\eta) = \sum_{\substack{ y_1, y_2 \in B \\ y_1- y_2 = \eta}} a(x_1)a(x_2)$$  
where we let $\mathcal{A}(\xi):=\mathcal{A}_{1}(\xi)$. With this notation, we  may rewrite the above as
$$ \leq |A|^{1/2}|B|^{1/2} \max_{\lambda} \left( \sum_{ \xi, \eta \in F^n} \mathcal{A}_\lambda (\xi) \mathcal{B}(\eta) e\left( \xi \cdot \eta \right)  \right)^{1/4}  $$

$$\leq |A|^{1/2}|B|^{1/2}  \left(  \left(\sum_{\xi \in F^n} |\mathcal{A}(\xi)|^2 \right)^{1/2}  \left( \sum_{\xi \in F^n } | \sum_{ \eta \in F^n} \mathcal{B}(\eta) e\left( \xi \cdot \eta \right)|^{2}  \right)^{1/2}    \right)^{1/4}.$$
Applying Parsavel and noting that $\sum_{\xi \in F^n}|\mathcal{A}(\xi)|^{2} \leq  \Lambda(A)$ gives
$$\leq |A|^{1/2} |B|^{1/2 } |F|^{n/8} \left( \Lambda(A) \Lambda(B) \right)^{1/8}. $$
This completes the proof.
\end{proof}
\begin{remark}The referee has pointed out that nearly the same lemma appears in \cite{BG}. \end{remark}

Combining Lemma \ref{lem:reduct} and Lemma \ref{lem:expSum} gives us the following:
\begin{proposition}\label{prop:energyToExtract}Let $n > d$ and $M: F^d \rightarrow F^n$. Let $\eta >0 $ and assume that for every subset $A \subseteq  F^d$ with $|A| \sim |F|^{\frac{n}{(8-2\alpha )}}$ one has the energy estimate $\Lambda( M(A)) \lesssim |A|^{\alpha}$. Then the map $(x,y) \rightarrow \rho \left( M(x)\times M(y) \right) $ is an extractor with min-entropy rate near $\frac{n}{d(8-d\alpha )}.$
\end{proposition}

At this point we state the following additive energy estimates of Rudnev and Shkredov \cite{RS} for subsets of the $3$-dimensional and $4$-dimensional paraboloids, $P_3$ and $P_4$.
\begin{theorem}We have the following additive energy estimates. Let $F$ be a prime order field in which $-1$ is not a square. $A \subset P_{3}$ with $|A| \leq |F|^{\frac{26}{21}}$, then 
$$\Lambda(A) \lesssim |A|^{\frac{17}{7}}.$$
Let $F$ be an arbitrary prime order finite field. Let $B \subset P_4$ with $p^{4/3} \leq |B| \leq p^{2}$, then 
$$\Lambda(B) \lesssim |B|^{\frac{5}{2}}.$$
\end{theorem}Inserting these estimates into Proposition \ref{prop:energyToExtract} proves Theorem \ref{thm:3d} and Theorem \ref{thm:4d}. \textbf{Note Added:} Recently the author has slightly improved the exponent in the first energy estimate above from $17/7$ to $99/41+\epsilon$ for any $\epsilon >0$. This improves the min-entropy rate of the respective extractor to near $123/260$ from $21/41$. See \cite{LewkoRect}.

\texttt{M. Lewko}

\textit{mlewko@gmail.com}

\end{document}